\documentclass{article}

\usepackage{latexsym}
\usepackage{amsthm}
\usepackage{amsmath}

\title{Periodic attractors of perturbed one dimensional maps}
\author{O. Kozlovski}

\newtheorem{mytheorem}{Theorem}
\newtheorem*{theorema}{Theorem}

\newtheorem{theorem}{Theorem}[section]
\newtheorem{proposition}[theorem]{Proposition}
\newtheorem{lemma}[theorem]{Lemma} 

\newtheorem{conjecture}{Conjecture}

\theoremstyle{definition}

\newenvironment{theorem-proof}[1]
{\underline{Proof of Theorem~\ref{#1}} \newline}
{\hfill {$\Box$}}

\newfont{\Bb}{msbm10}
\newcommand{\N}{\mbox{\Bb N}}

\newcommand{\R}{\mbox{\Bb R}}

\newcommand{\interior}{\mathop{\mbox{int}}}

\newcommand{\Or}{\cal O}
\newcommand{\Fc}{{\cal F}}
\newcommand{\Nc}{{\cal N}}
\newcommand{\Sc}{{\cal S}}

\begin{document}

\maketitle

\begin{abstract}
  In this paper we investigate how many periodic attractors maps in a
  small neighbourhood of a given map can have. For this purpose we
  develop new tools which help to make uniform cross-ratio distortion
  estimates in a neighbourhood of a map with degenerate critical
  points.
\end{abstract}

\section{Introduction}
\label{sec:introduction}

Let $\Nc$ denote an interval or a circle and let $f: \Nc \to \Nc$ be a
$C^\infty$ map.  In this paper we use the standard notions of a
periodic point of $f$, its period, its immediate basin of attraction,
etc... One can find all relevant definitions in \cite{Melo1995}.

The map $f$ can have infinitely many periodic attractors, however this
is a non generic situation: if all critical points of $f$ are
non-flat, the periods of the periodic attracting orbits are bounded
from above, therefore if $f$ has infinitely many periodic attracting
orbits, they should accumulate on neutral periodic orbits and the
periods of these neutral orbits are also bounded, see \cite{M-dM-S} or
\cite{Melo1995}, Theorem B, p. 268. If $f$ has a flat critical point,
the periods of periodic attractors are not necessary bounded, an
example of such a map is given in \cite{Kaloshin2011}. As usual, we
call a critical point $c$ \emph{non-flat} if in a neighbourhood of $c$
the function $f(x)$ can be written as $\pm (\phi(x))^d$ where $\phi$
is a diffeomorphism and $d\in \N$, $d\ge 2$. For $C^\infty$ maps it is
equivalent to $D^df(c) \neq 0$ for some $d\ge 2$.

In this paper we will study whether a small perturbation of $f$ can
have infinitely many periodic attractors and related questions. The
simple answer to this problem is ``yes'', one can construct an example
of a $C^\infty$ map $f$ with a quadratic critical point which has a
finite number of periodic attractors such that in any $C^\infty$
neighbourhood of $f$ there are maps which have infinitely many
periodic attracting points and the periods of these points can be
arbitrarily large, see \cite{K12}. The source of these attractors is a
parabolic fixed point, and our first theorem shows that if $f$ does not
have neutral periodic orbits and all critical points of $f$ are
quadratic, then this phenomenon of having unbounded number of
attractors for maps in an arbitrarily small neighbourhood of $f$ is not
possible.

\begin{mytheorem}\label{thr:A}
  Let $f:\Nc \to \Nc$ be a $C^3$ map with quadratic critical
  points. Suppose that $f$ does not have neutral periodic orbits.

  Then there exists a neighbourhood $\Fc \subset C^3(\Nc)$ of $f$ and a
  natural number $n_0$ such that if $g\in \Fc$ and $\Or$ is an
  attracting periodic orbit of $g$, then either the period of the
  orbit $\Or$ is less than $n_0$ or there exists a critical point $c$
  of $g$ whose iterates converge to $\Or$ under iterations of the map
  $g$.

  In particular, all maps in $\Fc$ have finitely many periodic
  attractors and the number of these attractors is bounded by the
  number of attractors of $f$ plus the number of critical points of
  $f$.
\end{mytheorem}

In \cite{Melo1995}, Theorem B', p. 268, a stronger statement is
stated: though the conclusion in the statement is similar to
Theorem~\ref{thr:A}, it is not required that $f$ has no neutral
periodic points and it is not required that all critical points of $f$
are quadratic. As the example above shows this statement is not
correct and the authors of \cite{Melo1995} issued an erratum shortly
after the book was published.

So, we see that there are situations when small perturbations of $f$
can create unbounded number of periodic attractors. If $f$ has
quadratic critical points, it is possible to prove that this is not
typical. More precisely, the following is proven in \cite{K12}: let
$\Sc$ be a space of $C^d$, $d\ge 3$, maps of $\Nc$ with all critical
points quadratic and exclude diffeomorphisms of the circle from $\Sc$;
then for a generic smooth family $f_\lambda$ of maps in $\Sc$ there
exists $M>0$ such that the number of periodic attracting orbits of any
map in this family $f_\lambda$ is bounded by $M$. Interestingly enough
for a generic non trivial smooth family of circle diffeomorphisms such
a bound does not exists, i.e. there are maps in a generic family with
arbitrarily large number of periodic attracting orbits.

The situation gets significantly more complicated if we allow
degenerate (but non-flat) critical points. By a degenerate non-flat
critical point we mean a point $c$ of $f$ such that $Df^k(c)=0$ for
$k=1,2,\ldots,m-1$ and $Df^m(c) \neq 0$, where $m \ge 3$. 

Let us construct an example showing that Theorem~\ref{thr:A}
does not hold if we allow degenerate critical points of the map
$f$. Let $f \in C^\omega$ be a map of a circle topologically equivalent to the
doubling map $x \mapsto 2x \mod(1)$. Moreover, suppose that $f$ has
one critical point $c$ of cubic type and the orbit of $c$ is dense. 
Then there are maps arbitrarily close to $f$ in $C^\omega$ topology
such that they still have a cubic critical point and its iterates are
attracted to a periodic attracting orbit of high period. We can
perturb these maps so that the obtained maps do not have critical
points at all, but still have periodic attractors. Thus, for any $n_0$
we can find a map $g$ arbitrarily close to $f$ which has a periodic
attracting orbit of period larger than $n_0$ and no critical
points. This obviously contradicts the first part of the conclusion of
Theorem~\ref{thr:A}.

One might think that if a map has critical points of even degree,
then examples like above are impossible because critical points of
even degree cannot be destroyed by a small perturbation. Let us sketch
an example showing that this is not the case. Let $f \in C^\omega$ be
a unimodal map with a critical point $c$ of degree 4 such that
$a=f^{n_0}(c)$ is a repelling periodic point for some $n_0$ (so the map $f$ is
Misiurewicz). There exist an interval $J_0$ containing the critical
point $c$, a sequence of intervals $J_k$,
$k=1,2,\ldots,$ such that $J_k \to a$ as $k \to \infty$, and a sequence
$n_k$ such that $f^{n_k}(J_k)=J_0$ and $f^{n_k}|_{J_k}$ is a
diffeomorphism for all $k$. Moreover, under small perturbations of $f$
the repelling periodic point $a$ and the intervals $J_k$ persist,
i.e. if $g$ close enough to $f$, there exist a repelling periodic
point $a_g$ of $g$ close to $a$ and of the same period, and intervals
$J_{g,k}$ such that $\lim_{k \to \infty}J_{g,k}= a_g$,
$f^{n_k}(J_{g,k})=J_0$ and $f^{n_k}|_{J_{g,k}}$ is a diffeomorphism.
Using these intervals $J_k$ we can construct a sequence $g_{1,k}$ of
perturbations of $f$ in such a way that every map $g_{1,k}$ has two
critical points $c^2_{1,k}$ and $c^3_{1,k}$ of degrees 2 and 3 such
that the quadratic critical point $c^2_{1,k}$ is still mapped to
$a_{g_{1,k}}$ by $g^{n_0}_{1,k}$ and the cubic critical point becomes
a superattractor so that $g^{n_0}_{1,k}(c^3_{1,k})\in J_{g_{1,k},k}$ and
$g^{n_0+n_k}_{1,k}(c^3_{1,k})=c^3_{1,k}$. In the same way we can
perturb each of $g_{1,k}$ and obtain maps $g_{2,k}$ which still have
two critical points of degree 2 and 3, their cubic critical point are
still superattractors of period $n_0+n_k$ and the quadratic critical
points become  superattractors as well. Finally we can brake cubic
critical points of maps $g_{2,k}$ and obtain a sequence of maps
$g_{3,k}$ which satisfies the following properties: $\lim_{k\to \infty}
g_{3,k}=f$, $g_{3,k}$ are unimodal maps with quadratic critical points,
every map $g_{3,k}$ has two periodic attractors and periods of these
attractors tend to infinity as $k \to \infty$. Again, this contradicts
the conclusion of Theorem~\ref{thr:A}.

These examples show that a degenerate critical point of $f$ can
disappear under a perturbation or loose its degree, but the perturbed
map $g$ can have a periodic attractor related to this disappeared
critical point. We conjecture that the second part of
Theorem~\ref{thr:A} holds in this case:

\begin{conjecture}
  Let $f:\Nc \to \Nc$ be a $C^3$ map with non-flat critical points.
  Suppose that $f$ does not have neutral periodic orbits.

  Then there exists a neighbourhood $\Fc \subset C^3(\Nc)$ of $f$ such
  that for any $g\in \Fc$ the number of periodic attractors of $g$ is
  bounded by the number of attractors of $f$ plus the number of
  critical points of $f$ counted with their multiplicities.
\end{conjecture}

By definition the \emph{multiplicity} of a critical point $c$ is $m-1$
where $m$ is such that $Df^k(c)=0$ for $k=1,2,\ldots,m-1$ and $Df^m(c)
\neq 0$.

We have already mentioned that this conjecture does not hold if we
allow neutral periodic orbits for the map $f$ and there might be no
upper bound on the number of attractors for maps close to $f$.  The
next theorem shows that nevertheless we can group these attractors in
such a way that periodic attracting orbits in the same group are
related to each other in a very simple way and there is a uniform
bound on the number of these groups. To state this result we need a
few definitions first.

If $p$ is a periodic point of $f$ and $n$ is its period, then we
will call the number $2n$ the \emph{orientation preserving period} of
$p$ if $Df^n(p)<0$, and if $Df^n(p)\ge 0$, then the \emph{orientation
  preserving period} of $p$ is just $n$.

We will call a closed interval $I\subset \Nc$ \emph{periodic} if there
is $n\in \N$ such that $f^n(I)=I$ and $f^n : I \to I$ is a bijection.
Any periodic interval $I$ of period $n$ contains one or more periodic
points of period $n$ and if $f^n|_I$ is orientation reversing, it can
contain periodic points of period $2n$. If $n$ is even, $I$ can
contain a periodic point of period $n/2$ in its boundary. The interval
$I$ cannot contain periodic points of any other periods except
$n,2n,n/2$.

A \emph{pack of periodic points} is a collection of periodic points
such that they all belong to some closed periodic interval (maybe,
degenerate) and there is no larger periodic interval which contains
more periodic points.  A pack can consist of just one periodic point.
All periodic points in a pack either have the same period, or there is
one periodic point of period $n$ which is orientation reversing and
all other periodic points in the pack have period $2n$. In other
words, the orientation preserving period of all periodic points in a
pack is the same.  To every pack of periodic points one can associate
a \emph{pack of periodic orbits} in an obvious way.

This is the main result of the paper:
\begin{mytheorem}\label{thr:B}
  Let $f:\Nc \to \Nc$ be a $C^\infty$ map with non-flat critical points.
  There exist a neighbourhood $\Fc \subset C^\infty$ of $f$ and
  $M>0$, $\rho>0$ such that for any $g \in \Fc$ there exist at most
  $M$ exceptional packs of periodic orbits such that if $p$ is a
  periodic point of $g$ which is not a member of any of these
  exceptional packs, then
  $$
  |Dg^n(p)| > 1+\rho,
  $$
  where $n$ is a period of $p$.
\end{mytheorem}

In other words, in the neighbourhood of $f$ maps can possibly have many periodic
attractors, but the set of the periods of these attractors has a
uniformly bounded cardinality.

This theorem is stated for $C^\infty$ maps. The only place where it is
used is in the proof of Proposition~\ref{co-excep} where a result of
\cite{Sergeraert1976} is used. One can state this theorem for $C^k$
maps, however in this case extra conditions should be put on the
multiplicities of the critical points of the map $f$.

\section{Idea of the proofs and discussion}
\label{sec:idea-proof-theorem}

Let us discuss the main problems which arise when we want to carry
over some properties of a map $f$ to the maps in a small neighbourhood
of $f$. We will mainly keep in mind the following three results
closely related to Theorems~\ref{thr:A} and \ref{thr:B}: the Singer
theorem about periodic attractors of maps with negative Schwarzian
derivative \cite{singer}, a theorem about the Schwarzian derivative of
the first entry map to a small neighbourhood of a critical value
\cite{Kozlovski2000}, \cite{VanStrien2004}, and the Theorem B of
\cite{Melo1995}, p. 268 which we have already mentioned several times.
Let us remind the reader that the Schwarzian derivative of a function
$f$ is defined as $Sf(x)=\frac{D^3f(x)}{Df(x)} - \frac 32
\left(\frac{D^2f(x)}{Df(x)}\right)^2$. We will review some of the
properties of the Schwarzian derivative in
Section~\ref{sec:cross-ratio-estim}.

The maps we consider in this paper do not have wandering intervals and
one of the consequences of this fact is the ``Contraction principle'':
for every $\epsilon>0$ there exists $\delta>0$ such that if $J$ is an
interval with $|J|<\delta$ and not intersecting the immediate basin of
a periodic attractor, then for any $n>0$ each component of $f^{-n}(J)$
has length less than $\epsilon$. Of course, this statement holds for
maps in a neighbourhood of the map $f$, but then $\delta$ can depend
on the choice of the map, and, in general, one cannot have a
\emph{uniform} version of this statement.  On the other hand, this is
an important lemma in the proof of Theorem B of \cite{Melo1995} (see
Lemma 10.3.2, p.323) and in the proof of the fact that the first
return map to a small interval around a critical value has negative
Schwarzian derivative.

If one examines the proof of the contraction principle, it will be
apparent that the only obstruction to the proof of its uniform version is the
existence of parabolic points of $f$:

\begin{lemma}[Uniform Contraction Principle]
  \label{lm:unifor} 
  Let $f$ be a $C^1(\Nc)$ map and assume that $f$ does not have
  wandering intervals and neutral periodic points. Then for any
  $\epsilon>0$ there exist a neighbourhood $\Fc\subset C^1(\Nc)$ of
  $f$ and $\delta>0$ such that if $g \in \Fc$ and $J$ is an interval
  with $|J|<\delta$ and not intersecting the immediate basin of a
  periodic attractor of the map $g$, then for any $n>0$ each component
  of $g^{-n}(J)$ has length less than $\epsilon$.
\end{lemma}
The proof of this lemma is not hard and is given in Appendix.

Using the Uniform Contraction Principle one can show that the first
return map of $g$ to a small interval around a critical value has
negative Schwarzian derivative and the size of this interval is
uniformly bounded from below:

\begin{theorem}
 \label{thr:sfn}
 Let $f$ be a $C^3(\Nc)$ map of an interval or circle with quadratic
 critical points. Suppose that $f$ does not have neutral periodic
 orbits.  Let $c$ be a critical point of $f$ whose iterates do not
 converge to a periodic attractor.

 Then there exists a neighbourhood $\Fc \subset C^3(\Nc)$ of $f$ and a
 neighbourhood $J$ of $c$ such that if $g \in \Fc$ and $g^n(x)\in J$
 for some $x\in \Nc$ and $n\ge 0$, then $Sg^{n+1}(x)<0$.
\end{theorem}

The proof of this theorem follows the same lines as the proof of its
single map version, see \cite{Kozlovski2000} and \cite{VanStrien2004}.
One should notice that if all critical points of $f$ are quadratic,
one can choose a neighbourhood of the critical points so that the
Schwarzian derivative of a perturbed map $g$ will be negative with a
uniform estimate on it (see Appendix). In particular this implies that the
cross-ratio distortion estimates similar to \cite{Melo1989},
Theorem~1.2 hold uniformly. We will see that this is not the case if
$f$ has degenerate critical points. 

Now the proof of Theorem~\ref{thr:A} is straightforward
consequence of the Singer and Ma\~n\'e theorems.

\begin{theorem-proof}{thr:A}
  Take a neighbourhood $U$ of all critical points of $f$ whose
  iterates do not converge to periodic attractors of $f$ and so small
  that Theorem~\ref{thr:sfn} holds, i.e. if $J$ is a connected
  component of this neighbourhood, $g\in \Fc$, and $g^n(x) \in J$,
  then $Sg^{n+1}(x)<0$. We can also assume that boundary points of
  each connected component of $U$ are some preimages of repelling
  periodic points of $f$. Decreasing $\Fc$ if necessary we can assume
  that these periodic repellers persist for maps in $\Fc$ and, thus,
  the set $U_g$ can be defined so that the boundary points of $U_g$
  are preimages of some repellers of $g$ and continuously depend on
  $g$, $U_f=U$, and $Sg^{n+1}(x)<0$ if $g^n(x)\in U_g$. Let $W \subset
  U$ be a smaller  neighbourhood of critical points of $f$ and again
  by decreasing $\Fc$ we can assume that $W \subset U_g$ for all $g\in
  \Fc$.

  Let $\Or$ be an attracting periodic orbit of $g$ of period $n$ which
  intersects $W$. Let $p\in W\cap \Or$, $J$ be a connected
  component of $U_g$ containing $p$, and $R: X \to J$ be the first entry map of $g$
  to $J$.  The immediate basin of attraction $B$ of $g(p)$ cannot contain
  preimages of repelling periodic points, therefore it is entirely
  contained in a connected component of $X$. This implies that
  $Sg^n(x)<0$ for all $x \in B$ and Singer's argument shows that there
  is an iterate of a critical point of $g$ in $B$.

  The Ma\~n\'e theorem \cite{mane85} states that the set of points whose iterates under
  the map $f$ never entry the domain $W$ consists of a hyperbolic set,
  and attraction basins of non degenerate periodic attracting orbits
  (because $f$ does not have neutral periodic points). Thus, for small
  perturbations of $f$ the number of periodic attractors whose orbits
  do not intersect $W$ does not change.
\end{theorem-proof}
\vspace{2mm}

The statement about the negative Schwarzian derivative of the first
return map for maps in the neighbourhood of $f$ holds only if all
critical points of $f$ are quadratic.

Indeed, consider the function $\phi(x)=x^3$. This function has
negative Schwarzian derivative everywhere and, moreover, the
Schwarzian derivative of $\phi$ tends to minus infinity when $x$ goes
to zero.

Now consider small perturbations of $\phi$ of the form $\phi_\lambda
(x) = x^3+\lambda x$ where $|\lambda|\ll 1$. The Schwarzian derivative
of $\phi_\lambda$ is
$$
S\phi_\lambda(x) = 6 \frac{\lambda - 6x^2}{(\lambda+3x^2)^2}.
$$
We see that for small values of $\lambda$ at zero the Schwarzian
derivative is $6/\lambda$, thus it is positive and very large and
Theorem \ref{thr:sfn} cannot possibly hold if we drop the condition on
the critical points to be quadratic.

In fact, the cross-ratio distortion estimates we have mentioned above
also do not hold uniformly if we allow degenerate critical points. To
deal with this problem we will introduce a notion of the critical
intervals in Section~\ref{sec:cross-ratio-estim}. These critical
intervals will capture some properties of the critical points when
they cease to exist under a perturbation of the map. In particular, we
will show that the attracting periodic points of sufficiently high
period must have either a critical point or a definite part of a
critical intervals in their basin of attraction. This will be the main
step in proving Theorem~\ref{thr:B}.

Another application of the critical intervals is given in
Section~\ref{sec:unif-pullb-estim} where we prove a uniform version of
the pullback estimates widely used in the literature. These estimates
are also an important part in the proof of Theorem~\ref{thr:B}. Since
they might be independently useful and important in their own right we
state them here. See Section~\ref{sec:unif-pullb-estim} for more details.

\begin{theorema}[\ref{thr:pb}]
  Let $f$ be a $C^\infty(\Nc)$ map with all critical points non-flat.
  There exists a neighbourhood $\Fc$ of $f$ in $C^\infty(\Nc)$ and a
  function $\rho(\epsilon,N)$ such that the following holds.

  Let $g$ be in $\Fc$, $J\subset T$ be intervals such that $g^m|_T$ is a
  diffeomorphism and the intersection multiplicity of the intervals
  $g^k(T)$, $k=0,\ldots,m-1$, is bounded by $N$. Then
  $$
  D(T,J)<\rho (D(g^m(T),g^m(J)),N),
  $$
  where $D(T,J)=\frac{|T||J|}{|L||R|}$ denotes the cross-ratio.

  Moreover, $\rho (\epsilon, N)$ tends to zero when $\epsilon$ goes
  to zero and $N$ is fixed.
\end{theorema}

\begin{theorema}[\ref{thr:cr}]
  Let $f$ be a $C^\infty(\Nc)$ map with all critical points non-flat.
  There exists a neighbourhood $\Fc$ of $f$ in $C^\infty(\Nc)$ and a function
  $\rho(\epsilon,N)$ such that the following holds.

  Let $g$ be in $\Fc$, $\{J_k\}_{k=0}^m$ and $\{T_k\}_{k=0}^m$ be chains such that 
  $J_k\subset T_k$ for all $0\leq k\leq m$.
  Assume that the intersection multiplicity of $\{T_k\}_{k=0}^s$ is at most $N$ and 
  that $T_m$ contains an $\epsilon$-scaled neighbourhood of
  $J_m$. Then $T_0$ contains $\rho(\epsilon,N)$-scaled neighbourhood
  of $J_0$.

  Moreover, $\rho (\epsilon, N)$ tends to infinity when $\epsilon$ goes
  to infinity and $N$ is fixed.
\end{theorema}

\section{Cross-ratio estimates in the presence of large positive Schwarzian}
\label{sec:cross-ratio-estim}

There are many well known estimates for the cross-ratio distortion of
a map, however often these estimates involve constants which
implicitly depend on the map. In this section we will give a few
explicit estimates for the cross-ratio distortion.  First, we start
with the standard definitions of the cross-ratio and state a few of its
well-known properties.

Let $J\subset T$ be two intervals and $L$ and $R$ are connected
components of $T\setminus J$. The cross-ratio of these intervals is
defined as
$$
D(T,J)=\frac{|T||J|}{|L||R|}.
$$
If $f:T\to \R$ is monotone on $T$, the cross-ratio distortion of $f$
we define by
$$
B(f,T,J)=\frac{D(f(T),f(J))}{D(T,J)}.
$$

Let $f$ be a real differentiable function and $\{T_j\}_{j=0}^m$ be a
collection of intervals. The {\it intersection multiplicity} of
$\{T_j\}_{j=0}^m$ is the maximal number of intervals with a non-empty
intersection. The {\it order} of $\{T_j\}_{j=0}^m$ is the number of
intervals containing a critical point of $f$.  This sequence of
intervals $\{T_j\}_{j=0}^m$ is called a {\it chain} if $T_{j}$ is a
connected component of $f^{-1}(T_{j+1})$.

If $I$ is a real interval of the form $(a-b,a+b)$ and $\lambda>0$ then
we define $\lambda I=(a-\lambda b,a+\lambda b)$.  By definition
$(1+2\delta)I$ is called the $\delta$-scaled neighbourhood of $I$. We
say that $I$ is $\delta$-well-inside $J$ if $J\supset (1+2\delta)I$.

Let $f$ be a $C^3$ mapping. The Schwarzian derivative of $f$ is
defined as
$$
Sf(x) = \frac{D^3f(x)}{Df(x)} - \frac 32
\left(\frac{D^2f(x)}{Df(x)}\right)^2.
$$
It is easy to check that the Schwarzian derivative of a composition of
two maps has this form:
$$
S(f\circ g)(x)=Sf(g(x)) Dg(x)^2 + Sg(x).
$$
This formula implies that if a map has negative Schwarzian derivative
then all its iterates also have negative Schwarzian derivatives.

It is also well known that maps having negative Schwarzian derivative
increase cross-ratios. Next lemma gives an estimate on the cross-ratio
distortion in terms of the map's Schwarzian derivative.

\begin{lemma}\label{lm-cos2}
  Let $f:T \to f(T)$ be a  $C^3$ diffeomorphism and suppose that
  $Sf(x)< C$ for all $x$ in $T$ and some constant $C>0$. Moreover,
  suppose that $C|T|^2 < \frac{\pi^2}2$. Then for 
  any $J \subset T$ we have
  $$
  B(f,T,J)> \cos^2(\sqrt{C/2} |T|).
  $$
\end{lemma}

\emph{Remark.} One does need a bound on the size of the interval (as
in the lemma) in order to have a non trivial estimate on the
cross-ratio distortion from below. More precisely, for any
$\epsilon>0$ there exists a $C^3$ diffeomorphism $f:[0,1] \to [0,1]$
and an interval $J \in [0,1]$
such that $Sf(x) < {\pi^2}$ and $B(f,[0,1],J)<\epsilon$.

\begin{proof}
  First, using rescaling we can assume that $T=[0,1]$. Let $J=[a,b]$.
  The Schwarzian derivative of a Mobius transformation is zero,
  therefore post-composing the map with a Mobius transformation does
  not change the cross-ratio distortion $B(f,T,J)$ and the map's
  Schwarzian derivative.  By post-composing the map $f$ with an
  appropriate Mobius transformation we can assume that $f(0)=0$,
  $f(a)=a$ and $f(1)=1$. Since $f$ is monotone, we can now assume that
  $Df(x)> 0$.

  The interval $[0,a]$ is mapped onto itself by $f$, therefore there
  exists a point $u_1 \in [0,a]$ such that $Df(u_1)=1$. If $f(b)\ge
  b$, then $B(f,T,J)\ge 1$ and we are done. Otherwise, there are
  points $v_1\in [a,b]$, $v_2 \in [b,1]$ such that
  $$
  Df(v_1)=\frac{f(b)-a}{b-a}<1
  \quad \mbox{and} \quad
  Df(v_2)=\frac{1-f(b)}{1-b} > 1.$$
  Hence,
  there exists a point $u_2 \in [v_1,v_2]$ such that
  $Df(u_2)=1$. Notice that $B(g,T,J)=Df(v_1)/Df(v_2)$ and in order to
  estimate the cross-ratio distortion from below we should estimate
  $Df(v_1)$ from below and $Df(v_2)$ from above.

  By a direct computation one can check that another form for the
  Schwarzian derivative of $f$ is
  $$
  Sf(x) = -2 \sqrt{Df(x)} D^2 \frac 1{\sqrt{Df(x)}}.
  $$
  This implies that if we denote $\frac 1{\sqrt{Df(x)}}$ by $\phi(x)$, then
  the function $\phi$ satisfies the linear second order differential
  equation 
  \begin{eqnarray}
    \label{eq:12}
    \phi''(x)=-\frac 12 Sf(x) \phi(x).
  \end{eqnarray}
  Moreover, we know that 
  \begin{eqnarray}
    \phi(u_1)=\phi(u_2)=1.\label{eq:13}
  \end{eqnarray}
  Let us compare the solutions of this equation with the solutions of
  the equation
  \begin{eqnarray}
    \psi''(x)=-\frac 12 C \psi(x)
    \label{eq:14}
  \end{eqnarray}
  with the same boundary conditions 
  \begin{eqnarray}
    \psi(u_1)=\psi(u_2)=1.\label{eq:15}
  \end{eqnarray}
  
  \underline {Claim.} Suppose $\phi:[0,1] \to \R$ satisfies
  equations (\ref{eq:12}), (\ref{eq:13}) and $\phi(x)>0$ for all $x\in
  [0,1]$. Suppose that $\psi:[0,1] \to \R$ satisfies (\ref{eq:14}),
  (\ref{eq:15}). Then for all $x \in [u_1,u_2]$ one has
  $\phi(x)\le\psi(x)$ and for all $x\in [u_2,1]$ one has
  $\phi(x)\ge\psi(x)$.

  To prove this claim let us first notice that the inequality
  $C<\frac{\pi^2}2$ implies that $\psi(x)\ge 0 $ for all $x\in [0,1]$.

  Next, one can easily check that $\phi$ and $\psi$ satisfy the Picone
  identity
  $$
  D\left(\phi(x)\psi(x)\, D\left(\frac{\phi(x)}{\psi(x)}\right)\right)
  =
  \frac 12 (C-Sf(x)) \phi(x)^2 +
  \left(D\phi(x)-\frac {\phi(x)}{\psi(x)}D\psi(x)\right)^2.
  $$
  Notice that the right hand side in the Picone identity is always positive.

  Set $x_0=\inf\{x \in [u_1,1]: \phi(x)>\psi(x)\}$. By continuity we
  get $\phi(x_0)=\psi(x_0)$ and $D(\phi(x)/\psi(x))|_{x=x_0} \ge
  0$. The Picone identity implies that for all $x>x_0$ one has
  $$
  \phi(x)\psi(x)D\left(\frac{\phi(x)}{\psi(x)}\right)
  >
  \phi(x_0)\psi(x_0)D\left(\frac{\phi(x)}{\psi(x)}\right)|_{x=x_0}.
  $$
  In particular, we get $D\left(\frac{\phi(x)}{\psi(x)}\right)> 0$,
  and, therefore, $\phi(x)/\psi(x)> \phi(x_0)/\psi(x_0)=1$ for all
  $x\ge x_0$. 

  If $x_0<u_2$, we would have $1=\phi(u_2)>\psi(u_2)=1$ which is not
  possible, so $x_0\ge u_2$ and we have proved the first part of the
  claim.
  
  To prove the second part of the claim we should notice that since 
  $\phi(u_2)=\psi(u_2)=1$ and $\phi(x)\le \psi(x)$ for $x \in
  [u_1,u_2]$ we get $D(\phi(x)/\psi(x))|_{x=u_2} \ge 0$. Using the
  Picone identity once more and arguing as
  before we conclude that $\phi(x) \ge \psi(x)$ for all $x \in
  [u_2,1]$ and the proof of the claim is finished.

  Using this claim we can estimate the cross-ratio distortion in terms of the
  function $\psi$:
  \begin{eqnarray*}
    B(f,T,J) = \frac{Df(v_1)}{Df(v_2)}
    = \left(\frac{\phi(v_2)}{\phi(v_1)}\right)^2
    > \left(\frac{\psi(v_2)}{\psi(v_1)}\right)^2.
  \end{eqnarray*}

  The solution of equations (\ref{eq:14}), (\ref{eq:15}) is
  $$
  \psi(x)=\frac{\cos\left(\sqrt{C/2}\left(x-\frac{u_1+u_2}2\right)\right)}{\cos\left(\sqrt{C/2}\left(\frac{u_1-u_2}2\right)\right)}.
  $$
  On the interval $[0,1]$ the function $\psi$ reaches its maximum at
  the point $\frac{u_1+u_2}2$ and its minimum at one of the boundary
  points $0$ or $1$. Hence,
  \begin{eqnarray*}
   \frac{\psi(v_2)}{\psi(v_1)}
   &\ge&
   \min\left\{\cos\left(\sqrt{C/2}\left(\frac{u_1+u_2}2\right)\right),
   \cos\left(\sqrt{C/2}\left(1-\frac{u_1+u_2}2\right)\right)\right\}\\
   &\ge&
   \cos(\sqrt{C/2})
  \end{eqnarray*}
  since $\frac{u_1+u_2}2 \in [0,1]$ and $\sqrt{C/2}<\frac \pi 2$.
\end{proof}

If the Schwarzian derivative is strictly negative, the cross-ratio
distortion is always greater than one. If it is negative and
bounded away from zero by some constant, in general, one cannot improve this estimate on
the cross-ratio distortion: the interval $J$ can be small and close to
one of the end points of the interval $T$. However, if $J$ is
situated exactly in the centre of $T$ and not very small, we can get a definite increase
of the cross-ratio:

\begin{lemma}\label{lm-sinh}
  Let $f:T \to f(T)$ be a  $C^3$ diffeomorphism and suppose that
  $Sf(x)< -C$ for all $x$ in $T$ and some constant $C>0$. Then for 
  any interval $J \subset T$ such that $T$ is equal to $\delta$-scaled
  neighbourhood of $J$ we have
  $$
  B(f,T,J)
  >
  \frac 
  {1+2\delta}
  {\sqrt{C/2} |T|} 
  \sinh\left(\frac{\sqrt{C/2} |T|}{1+2\delta}\right)
  \ge
  1+\frac 1{12}\frac{C |T|^2}{(1+2\delta)^2}.
  $$
\end{lemma}

\begin{proof}
  Start by rescaling $T$ to $[0,1]$. Then $J=[a,1-a]$, where
  $a=\frac{\delta}{1+2\delta}$. By post-composing $f$ with a Mobius
  transformation we can assume that $f(0)=0$, $f(1)=1$ and
  $f(a)+f(1-a)=1$. Since the Schwarzian derivative is negative on $T$
  we already know that $f(a)\le a$.

  Let us denote the ratio $f(a)/a$ by $r$. Notice that $(1-f(1-a))/a$
  is equal to $r$ as well and that $r\le 1$.  By the Roley theorem
  there exist points $u_1\in [0,a]$ and $u_2 \in [1-a,1]$ such that
  $$
  Df(u_1)=Df(u_2)=r.
  $$

  As in the proof of the previous lemma let us denote $\frac
  1{\sqrt{Df(x)}}$ by $\phi(x)$. The function $\phi$ satisfies
  equation~(\ref{eq:12}) with boundary conditions
  \begin{eqnarray}
    \label{eq:4}
    \phi(u_1)=\phi(u_2)=\frac 1{\sqrt{r}}.
  \end{eqnarray}
  We will compare the solution of this equation with the function
  $\psi$ which satisfies
  \begin{eqnarray}
    \label{eq:6}
    \psi''=\frac 12 C \psi
  \end{eqnarray}
  and the boundary conditions similar to (\ref{eq:4}). This equation is
  easy to solve and the solution is
  $$
  \psi(x)=\frac
  {\cosh\left(\sqrt{C/2}(x-\frac{u_1+u_2}2)\right)}
  {\sqrt{r}\cosh\left(\sqrt{C/2}(\frac{u_1-u_2}2)\right)}.
  $$

  As in the proof of the previous lemma the following is true:
  $\phi(x)\le\psi(x)$ for all $x\in [u_1,u_2]$.  Now, let us estimate
  the cross-ratio distortion
  \begin{align*}
    B(f,T,J)
    &=
    \frac{f(1-a)-f(a)}{(1-2a)r^2}\\
    &=
    \frac{1}{(1-2a)r^2} \int_{a}^{1-a} Df(t)\, dt\\
    &\ge
    \frac
    {\cosh\left(\sqrt{C/2}(\frac{u_1-u_2}2)\right)^2}
    {(1-2a)r}
    \int_{a}^{1-a}
    \frac{1}
    {\cosh\left(\sqrt{C/2}(t-\frac{u_1+u_2}2)\right)^2}
    \, dt\\
    &=
    \frac
    {\cosh\left(\sqrt{C/2}(\frac{u_1-u_2}2)\right)^2}
    {\sqrt{C/2}(1-2a)r}
    \left. 
      \tanh\left(\sqrt{C/2}(t-\frac{u_1+u_2}2)\right)
    \right|_{t=a}^{1-a}
  \end{align*}
  By an elementary consideration one can show that under the
  restrictions $u_1\in [0,a]$ and $u_2\in[1-a,1]$ the last expression
  achieves its minimum when $u_1=a$ and $u_2=1-a$. Thus,
  \begin{align*}
    B(f,T,J)
    &\ge
    2\frac
    {\cosh\left(\sqrt{C/2}\left(\frac12 -a\right)\right)^2}
    {\sqrt{C/2}(1-2a)}
    \tanh\left(\sqrt{C/2}\left(\frac12 -a\right)\right)\\
    &=
    \frac
    {\sinh\left(\sqrt{C/2}\left(1 -2a\right)\right)}
    {\sqrt{C/2}(1-2a)}
  \end{align*}
  
\end{proof}

In order to understand the cross-ratio distortion for maps in a
neighbourhood of a map which has degenerate critical point we first
study it in the case of the polynomials.

\begin{proposition} \label{thr:excep}
  For any polynomial $f$ of degree $d$ there exists at most $(d-1)/2$ 
  intervals $E_j$ (which we will call \emph{critical intervals}), $j=1,\ldots,d_E$ such that the following holds:

  \begin{itemize}
  \item 
    Let $\kappa \in (0,\frac 1{4\sqrt{d_E}})$ be a number,
    $T_1,\ldots,T_m$ be intervals and their intersection
    multiplicity be bounded by $N$. Moreover, suppose that
    $f|_{T_i}$ is a diffeomorphism and that
    $$
    |T_i \cap E_j| < \kappa |E_j|
    $$
    for all $i=1,\ldots,m$, $j=1,\ldots,d_E$.
    Then
    $$
    \prod_{i=1}^m B(f,T_i, J_i) > \exp(-16 \kappa N d_E^2),
    $$
    where $J_i \subset T_i$ are any intervals.
  \item 
    Let $\lambda > 1$, $\kappa \in (0,\frac 1{13\sqrt{d_E}})$ be some
    numbers,
    $J \subset T$ be intervals such
    that the interval $T$ is equal to the $\delta$-scaled
    neighbourhood of $J$ and $f|_T$ is a diffeomorphism. Moreover,
    assume that 
    \begin{align*}
      |T\cap E_j| &< \kappa |E_j| / \lambda
    \end{align*}
    for all $j=1,\ldots, d_E$ and either there exists a critical point
    $c$ of $f$ contained in the interval $\lambda T$ or there exists 
    $j_0 \in [1,d_E]$ such that 
    $$
    T\not \subset 2E_{j_0}
    \quad \mbox{and} \quad
    \lambda T \cap E_{j_0} \neq \emptyset.
    $$
    Then
    $$
    B(f,T,J)
    >
    1 +
    \frac 1{12}
    \left(
      \frac{16}{17(1+\lambda)^2}
      - 32 \frac{\kappa^2 d_E}{\lambda^2}
    \right)
    \frac 1{(1+2\delta)^2}.
    $$
  \end{itemize}
\end{proposition}

Notice that there is no dynamics involved in this proposition.

\begin{proof}
  The derivative of $f$ is also a polynomial and let $x_k$,
  $k=1,\ldots,d-1$ be its roots. Then the Schwarzian derivative of $f$
  can be written as
  \begin{eqnarray*}
    Sf(x)
    &=&
    2\sum_{1\le k<l\le d-1} \frac 1{(x-x_k)(x-x_l)} - \frac
    32 \left(\sum_{k=1}^{d-1} \frac 1{x-x_k}\right)^2 \\
    &=&
    -\sum_{k=1}^{d-1} \frac 1{(x-x_k)^2} - \frac 12 \left(\sum_{k=1}^{d-1} \frac 1{x-x_k}\right)^2.
  \end{eqnarray*}
  Let $a_j\pm ib_j$, $j=1,\ldots,d_E$, be all non real roots of $Df$
  among $x_1,\ldots, x_{d-1}$. Then the formula for the Schwarzian
  derivative above implies
  \begin{eqnarray*}
    Sf(x)
    &\le&
    -\sum_{j=1}^{d_E} \left( \frac 1{(x-a_j-ib_j)^2} +
      \frac 1{(x-a_j+ib_j)^2} \right) \\
    &=&
    -2\sum_{j=1}^{d_E} \frac{(x-a_j)^2-b_j^2}{((x-a_j)^2+b_j^2)^2}.
  \end{eqnarray*}

  Define the critical intervals as $E_j=[a_j-2b_j,a_j+2b_j]$. It is
  easy to see that if $x$ is a point which is not contained in any of
  the intervals $E_j$, then $Sf(x)<0$. Otherwise, let $E_j$ be an
  critical interval of minimal length containing the point $x$. The
  above estimate for the Schwarzian derivative implies that 
  $$
  Sf(x)< \frac{2d_E}{b_j^2}.
  $$
  
  If an interval $T_k$ is not contained in any of the critical
  intervals (but can have non empty intersection with them), then $T_k
  \cap [a_j-b_j,a_j+b_j]=\emptyset$ for all $j=1,\ldots,d_E$ because
  $|T_k \cap E_j|<\kappa |E_j|< |E_j|/4$, and therefore $Sf|_{T_k}<0$ and
  $B(f,T_k,J_k)>1$.

  Fix a critical interval $E_j$. Let $T_{k_1},\ldots,T_{k_{m'}}$ be
  all intervals which are contained in $E_j$ but are not contained in
  a critical interval of length smaller than $|E_j|$. We have
  already argued that $Sf|_{T_{k_i}} < \frac{2d_E}{b_j^2}$. By the
  choice of the number $\kappa$ we know that 
  \begin{eqnarray}
    \sqrt{\frac 12 \max_{x\in T_{k_i}} Sf(x)} \, |T_{k_i}|
    &<& \nonumber
    \frac{\sqrt{d_E}}{b_j} \kappa |E_j|\\
    &<& \label{eq:smain}
    \frac{\pi}2.
  \end{eqnarray}
  So, we can apply Lemma~\ref{lm-cos2} and get
  \begin{eqnarray*}
    \log B(f,T_{k_i},J_{k_i})
    &>&
    \log\left(\cos^2\left(\frac{\sqrt{d_E}}{b_j} |T_{k_i}|\right)\right)\\
    &>&
    - \frac{d_E}{b_j^2} |T_{k_i}|^2.
  \end{eqnarray*}
  Here we have used the fact that $\cos(x)\ge 1-x^2$ for all $x\in \R$.

  Now we can estimate the contribution of the cross-ratio distortions
  on all the intervals $T_{k_i}$.
  \begin{eqnarray*}
    \sum_{i=1}^{m'} \log B(f,T_{k_i},J_{k_i})
    &>&
    - \frac{d_E}{b_j^2} \sum_{i=1}^{m'} |T_{k_i}|^2\\
    &>&
    - \frac{d_E}{b_j^2} \kappa|E_j| \sum_{i=1}^{m'} |T_{k_i}|\\
    &>&
    - \frac{d_E}{b_j^2} \kappa {N|E_j|^2}\\
    &=&
    -16 \kappa {N d_E}.
  \end{eqnarray*}
  Thus, we get
  \begin{eqnarray*}
    \sum_{j=1}^m \log B(f,T_{k},J_{k})
    &>&
    -16 \kappa{N d_E^2}
  \end{eqnarray*}
  The first part of the proposition is proved, let us prove now the second
  part.

  First, suppose that we are in the first case, i.e. the exists a
  critical point $c$ such that $c \in \lambda T$. Set $I=[c,T]$. Since
  $c\in \lambda T$ we get $(1+\lambda)/2 \, |T| \ge |I|$.

  If $T$ is not contained in any critical interval, then arguing as
  before we get that $Sf(x) <0$ for all $x \in T$ and
  \begin{align*}
    \min_{x\in T} (-Sf(x)) \, |T|^2 
    &\ge
    \frac{|T|^2}{|I|^2}\\
    &\ge
    \frac 4{(1+\lambda)^2}.
  \end{align*}
  
  If $T$ is contained in some critical intervals, let
  $E_j$ be such an interval of minimal length. Using estimate
  (\ref{eq:smain}) and estimating the
  contribution to the Schwarzian derivative of critical intervals
  which contain $T$ we get
  \begin{align*}
    \min_{x\in T} (-Sf(x)) \, |T|^2 
    &\ge
    \frac 4{(1+\lambda)^2} - 32 \frac{\kappa^2 d_E}{\lambda^2}.
  \end{align*}
  Notice that 
  since $\kappa^2 < \frac 1{13^2 d_E}$ the right hand side of
  the inequality above is positive.

  Now consider the remaining case and assume that $\lambda T \cap
  E_{j_0} \neq \emptyset$ and $T \not \subset 2E_{j_0}$.

  Denote by $A$ the interval $[a_{j_0},T]$. Since $T \not \subset
  2E_{j_0} = [a_{j_0}-4b_{j_0},a_{j_0}+4b_{j_0}]$ we get
  $$
  |A| > 4 b_{j_0}.
  $$
  On the other hand, the condition
  $\lambda T \cap E_{j_0} \neq \emptyset$ implies
  $$
  |A|-(1+\lambda)/2 \, |T| < 2b_{j_0}.
  $$
  These two inequalities combined give an estimate on the length of
  the interval $|T|$:
  $$
  (1+\lambda)|T| > 4b_{j_0}.
  $$
  Another inequality we will be using which is easy to check is
  $$
  -2  \frac{(x-a)^2-b^2}{((x-a)^2+b^2)^2}
  \le
  -  \frac 1{(x-a)^2+b^2}
  $$
  if $|x-a|\ge 2b$.
  
  Using these inequalities we can get
  \begin{align*}
    \min_{x\in T} (-Sf(x)) \, |T|^2 
    &\ge
    \frac {|T|^2}{|A|^2+b_{j_0}^2}
    -32 \frac{\kappa^2 d_E}{\lambda^2}
  \end{align*}
  and let us estimate the term which contains $A$ and $T$:
  \begin{align*}
    \frac {|T|^2}{|A|^2+ b_{j_0}^2}
    &\ge
    \frac{|T|^2}{(2 b_{j_0} + (1+\lambda)/2 |T|)^2 + b_{j_0}^2}\\
    &=
    \frac 1{(2 b_{j_0}/|T| + (1+\lambda)/2 )^2 + (b_{j_0}/|T|)^2}\\
    &\ge
    \frac 1{((1+\lambda)/2 + (1+\lambda)/2 )^2 + ((1+\lambda)/4)^2}\\
    &=
    \frac {16}{17(1+\lambda)^2}.
  \end{align*}
  Applying Lemma~\ref{lm-sinh} to the obtained inequalities we
  finish the proof.
\end{proof}

\begin{proposition}\label{co-excep}
  Let $f$ be a $C^\infty(\Nc)$ map with all critical points non-flat.
  There exist a neighbourhood $\Fc$ of $f$ in $C^\infty(\Nc)$ and  $d_f \ge 0$ such that
  for any $\epsilon >0$, $N>0$, $\delta>0$, $\lambda >1$ there exist
  $\kappa>0$ and $\tau>0$ with the following properties.
  For any $g \in \Fc$ there exist at most $d_f$ critical intervals
  $E_j$, $j=1,\ldots,d_g$ such that
  \begin{itemize}
  \item 
    if $J \subset T$ are intervals,
    $g^m|T$ is a diffeomorphism, the intersection multiplicity of
    $\{g^k(T)\}_{k=0}^{m-1}$ is bounded by $N$, $|g^k(T)|<\kappa$, and $|g^k(T) \cap E_j|
    < \kappa |E_j|$ for all $k=0,\ldots,m-1$ and $j=1,\ldots,d_g$, then
    $$
    B(g^m,T,J)>1-\epsilon;
    $$
  \item 
    if $J$ is an interval, $T=(1+2\delta) J$, $g|_T$ is a
    diffeomorphism, $|T|<\kappa$,
    \begin{align*}
      |T\cap E_j| &< \kappa |E_j| / \lambda
    \end{align*}
    for all $j=1,\ldots, d_g$ and either there exists a critical point
    $c$ of $g$ contained in the interval $\lambda T$ or there exists 
    $j_0 \in [1,d_g]$ such that
    $$
    T\not \subset 2E_{j_0}
    \quad \mbox{and} \quad
    \lambda T \cap E_{j_0} \neq \emptyset.
    $$
    Then
    $$
    B(g,T,J)
    >
    1 + \tau.
    $$
  \end{itemize}

\end{proposition}

\begin{proof}
  Fix small neighbourhood $U$ of the critical set of $f$.
  Take a connected component $U_0$ of $U$. Decreasing $U_0$ if
  necessary we can assume that $U_0$ contains only one critical point
  of $f$ of order $d$. In the domain $U_0$ the function $f$ can be
  written as $f|_{U_0}=(\phi_0)^d$ where $\phi_0$ is a
  diffeomorphism. Take $\Fc$ small enough so that the function $g\in
  \Fc$ can be decomposed as $g|_{U_0}=\psi \circ P \circ \phi$ where
  $P$ is a polynomial of degree at most $d$ and $\psi$ and $\phi$ are
  diffeomorphisms so that $\psi$ is $C^\infty$ close to the identity
  map and $\phi$ is $C^\infty$ close to $\phi_0$, see
  \cite{Sergeraert1976}. So, the Schwarzian derivatives of $\psi$ and
  $\phi$ are uniformly bounded. Now we can apply Lemma~\ref{lm-cos2}
  to the functions $\phi$, $\psi$ and Proposition~\ref{thr:excep} to the
  polynomial $P$. 

  Take another neighbourhood $W$ of the critical set of $f$ so that
  $W$ is compactly contained in $U$.  Decrease $\Fc$ if necessary so
  that the Schwarzian derivative of maps in $\Fc$ is uniformly bounded
  from above outside $W$.  Then, Lemma~\ref{lm-cos2} implies that
  there are constants $C$ and $\kappa$ such that for all $g\in \Fc$
  $$B(g,g^k(T),g^k(J)) > 1-C |g^k(T)|^2$$
  if the interval $g^k(T)$ is disjoint from $W$ and $|g^k(T)|<\kappa$.

  Decrease $\kappa$ so that if an interval of length $\kappa$ has a
  non empty intersection with $W$, then this interval is contained in
  $U$. 

  Now we can estimate the cross-ratio distortion:
  \begin{align*}
    \log(B(g^m,T,J))
    &=
    \sum_{k=1}^{m-1} \log(B(g,g^k(T),g^k(J))) \\
    &=
    \Big(\sum_{g^k(T) \cap W=\emptyset} 
    +
    \sum_{g^k(T) \cap W\neq \emptyset}\Big)
    \log(B(g,g^k(T),g^k(J)))\\
    &> 
    -C \sum_{k=1}^{m-1} |g^k(T)|^2
    -16\kappa N d_f^2\\
    &>
    -C \kappa N |\Nc|-16\kappa N d_f^2.
  \end{align*}
  The last expression can be made arbitrarily close to zero by decreasing $\kappa$.
\end{proof}

\section{Uniform pullback estimates}
\label{sec:unif-pullb-estim}

We also want to know a bound from below on the cross-ratio distortion when
there are no bounds on the length of the intervals $g^k(T)$. Such a
bound exists though it is not as good as in the proposition
above. To prove this bound we need a few 
statements.

\begin{lemma}\label{lm-poly-pb}
  There exists a function $\rho(\epsilon,d)>0$ such that if $f$ is a
  polynomial of degree less or equal to $d$, $\hat J\subset \hat T$ are
  intervals, $\hat T$ contains $\epsilon$-scaled neighbourhood of
  $\hat J$,
  $T$ is a connected component of $f^{-1}(\hat T)$, $J$ is a
  connected component of $f^{-1}(\hat J)$ which is contained in $T$,
  then the interval $T$ contains $\rho(\epsilon,d)$-scaled
  neighbourhood of $J$.

  Moreover, $\rho (\epsilon,d)$ tends to infinity when $\epsilon$ goes
  to infinity with fixed $d$.
\end{lemma}

\begin{lemma}\label{lm-poly-cr}
  There exists a function $\rho(\epsilon,d)>0$ such that if $f$ is a
  polynomial of degree less or equal to $d$, $J\subset T$ are
  intervals, $f|_T$ is a diffeomorphism, then
  $$
  D(T,J) < \rho (D(f(T),f(J)),d).
  $$
  Moreover, $\rho (\epsilon,d)$ tends to zero when $\epsilon$ goes
  to zero with fixed $d$.
\end{lemma}

The second lemma is a straightforward consequence of the first one, we
will prove here only the first lemma.

\begin{proof}
  First, we can assume that $\hat T$ is equal to $\epsilon$-scaled
  neighbourhood of $\hat J$. Next, we can rescale the polynomial $f$
  and assume that $T=\hat T=[0,1]$. Thus, $f([0,1])\subset [0,1]$.

  Let $A_d$ be a set of polynomials of degree less or equal to $d$
  such that for any $g\in A_d$ one has $g(x)\in [0,1]$ for any $x\in
  [0,1]$ and $g(y)\in \{0,1\}$ for $y\in \{0,1\}$. The set $A_d$ is
  compact. Indeed, any polynomial $d$ of degree less or equal to $d$ is uniquely determined by
  its values at the points $x_k=\frac kd$, where $k=0,\ldots,d$. So, 
  $$
  g(x)=\sum_{k=0}^d g(x_k) N_k(x),
  $$
  where $N_k$ is a Newton polynomial
  $$
  N_k(x)=\frac{\prod_{i\neq k}(x-x_i)}{\prod_{i\neq k}(x_k-x_i)}.
  $$
  Since $g(x_k)\in [0,1]$ for $g \in A_d$ and all
  $k=0,\ldots,d$, we see that the set $A_d$ is compact. Therefore,
  the maximum of the derivatives of polynomials in $A_d$ is bounded:
  $$\inf_{g\in A_d}\max_{x\in [0,1]} |Dg(x)| < K.$$
  This implies that both components $T\setminus J$ are greater than
  $\frac{\epsilon}{K(1+2\epsilon)}$ and the function $\rho$
  exists. Using the compactness argument once again, it is easy to
  show that $\rho(\epsilon, d)\to \infty$ when $\epsilon \to \infty$.
\end{proof}

\begin{proposition}
  \label{co-poly-pd}
  Let $f$ be a $C^\infty(\Nc)$ map with all critical points non-flat.
  There exists a neighbourhood $\Fc$ of $f$ in $C^\infty$ and a
  function $\rho(\epsilon)$ such that the following holds.

  Let $g$ be in $\Fc$, $\hat J\subset \hat T$ are intervals, $\hat T$
  contains $\epsilon$-scaled neighbourhood of $\hat J$, $T$ is a
  connected component of $g^{-1}(\hat T)$, $J$ is a connected component of
  $g^{-1}(\hat J)$ which is contained in $T$, then the interval $T$
  contains $\rho(\epsilon)$-scaled neighbourhood of $J$.

  Moreover, $\rho (\epsilon)$ tends to infinity when $\epsilon$ goes
  to infinity.
\end{proposition}

\begin{proof}
The proof of this proposition is similar to the proof of
Proposition~\ref{co-excep}.

Fix two neighbourhoods $U\subset U'$ of the critical set and let each
connected component of $U$ contain just one critical point of $f$ and $U'$
compactly contains $U$. Take $\Fc$ so small that the distortion of the
derivative of maps $g\in \Fc$ on the complement to $U$ is bounded and
that inside every connected component of $U'$ a map $g\in \Fc$ can be
decomposed as $\psi\circ P \circ \phi$ where $P$ is a polynomial of
uniformly bounded degree and $\psi$, $\phi$ are diffeomorphism, see
the proof of Proposition~\ref{co-excep}. From the lemma above it follows
that if $T\subset U'$ then the function $\rho$ exists. Since the
derivative distortion on the complement of $U$ is uniformly bounded,
the function $\rho$ exists also when $T$ belongs to the complement of
$U$. In the remaining case the interval $T$ must contain a component
of $U'\setminus U$ and cannot be small. The set $\Fc$ is precompact in
the $C^1$ topology and using compactness argument again, we get the
function $\rho$ in the remaining case too.
\end{proof}

Similarly we can get

\begin{proposition}
  \label{co-poly-cr}
  Let $f$ be a $C^\infty(\Nc)$ map with all critical points
  non-flat.  There exists a neighbourhood $\Fc$ of $f$ in $C^\infty$ and a
  function $\rho(\epsilon)$ such that the following holds.

  Let $g$ be in $\Fc$, $J\subset T$ are intervals, $g|_T$ is a
  diffeomorphism, then
  $$
  D(T,J) < \rho (D(g(T),g(J))).
  $$
  Moreover, $\rho (\epsilon)$ tends to zero when $\epsilon$ goes
  to zero.
\end{proposition}

\begin{theorem} \label{thr:pb}
  Let $f$ be a $C^\infty(\Nc)$ map with all critical points non-flat.
  There exists a neighbourhood $\Fc$ of $f$ in $C^\infty(\Nc)$ and a
  function $\rho(\epsilon,N)$ such that the following holds.

  Let $g$ be in $\Fc$, $J\subset T$ be intervals such that $g^m|_T$ is a
  diffeomorphism and the intersection multiplicity of the intervals
  $g^k(T)$, $k=0,\ldots,m-1$, is bounded by $N$. Then
  $$
  D(T,J)<\rho (D(g^m(T),g^m(J)),N).
  $$
  Moreover, $\rho (\epsilon, N)$ tends to zero when $\epsilon$ goes
  to zero and $N$ is fixed.
\end{theorem}

\begin{proof}
  Let $\Fc$ be so small that Proposition~\ref{co-excep} holds with
  $\epsilon = \frac 12$, and Proposition~\ref{co-poly-cr} holds as well. Let
  $\kappa$ be the constant given by the first proposition and $\rho$ is a
  function given by the second one.
  Fix $g \in \Fc$ and let $E_1,\ldots,E_d$ be the corresponding
  critical intervals.

  Let $k_1<\ldots<k_{m'}$ be all indexes such that for every $k_i$
  either $|g^{k_i}(T)|\ge \kappa$ or there is $j$ such that
  $|g^{k_i}(T)\cap E_j|\ge\kappa |E_j|$.  If $k\neq k_{i}$, then
  $|g^k(T)|<\kappa$ and $|g^k(T)\cap E_j|<\kappa |E_j|$ for all
  $j=1,\ldots,d$, and the first part of Proposition~\ref{co-excep} can be applied to
  such intervals. Clearly, the number $m'$ of these indexes is bounded
  above by some constant which depends only on $\kappa$, the number of
  critical intervals $d$ and the intersection multiplicity $N$ (and
  independent of the choice of $g$).

  Due to Proposition~\ref{co-excep} we have
  $$
  D(g^{k_{m'}+1}(T),g^{k_{m'}+1}(J)) < 2 D(g^m(T),g^m(J)).
  $$
  Now we can apply Proposition~\ref{co-poly-cr} to the map $g:
  g^{k_{m'}}(T) \to g^{k_{m'}+1}(T)$ and get
  $$
  D(g^{k_{m'}}(T),g^{k_{m'}}(J)) < \rho(2 D(g^m(T),g^m(J)))
  $$
  Denote $\psi(D)=\rho(2D)$. Then we get
  $$
  D(T,J) < 2 \psi^{m'}(D(g^m(T),g^m(J)))
  $$
  and since $m'$ is uniformly bounded the theorem is proved.
\end{proof}

\begin{theorem} 
  \label{thr:cr}
  Let $f$ be a $C^\infty(\Nc)$ map with all critical points non-flat.
  There exists a neighbourhood $\Fc$ of $f$ in $C^\infty(\Nc)$ and a function
  $\rho(\epsilon,N)$ such that the following holds.

  Let $g$ be in $\Fc$, $\{J_k\}_{k=0}^m$ and $\{T_k\}_{k=0}^m$ be chains such that 
  $J_k\subset T_k$ for all $0\leq k\leq m$.
  Assume that the intersection multiplicity of $\{T_k\}_{k=0}^s$ is at most $N$ and 
  that $T_m$ contains an $\epsilon$-scaled neighbourhood of
  $J_m$. Then $T_0$ contains $\rho(\epsilon,N)$-scaled neighbourhood
  of $J_0$.

  Moreover, $\rho (\epsilon, N)$ tends to infinity when $\epsilon$ goes
  to infinity and $N$ is fixed.
\end{theorem}

\begin{proof}
  This time let $k_1<\cdots<k_{m'}$ be all indexes such that the
  interval $T_{k_i}$ contains at least one critical point. If $\Fc$ is
  small, the number of critical points of a map $g\in \Fc$ is uniformly bounded,
  and since the intersection multiplicity of the intervals
  $\{T_k\}_{k=0}^m$ is bounded by $N$, the number $m'$ is uniformly
  bounded as well.

  Now, we can apply the previous theorem to maps $g^{k_{i+1}-k_i+1}:
  g(T_{k_i}) \to T_{k_{i+1}}$ and  Proposition~\ref{co-poly-pd} to
  maps $g: T_{k_i} \to T_{k_{i+1}}$ and finish the proof.
\end{proof}

\section{Proof of Theorem \ref{thr:B}}
\label{sec:maps-with-degenerate}

The proof of this theorem uses the same ideas as in \cite{M-dM-S} or \cite{Melo1995}
but we will need to tweak that proof quite a bit. We will also follow
the notation in the book \cite{Melo1995} where possible.

We start the proof by making a few trivial observations. If $f$ is a
diffeomorphism of a circle, then the neighbourhood $\Fc$ of $f$ can be
taken so it consists only of diffeomorphisms. In this case the theorem
trivially holds as all periodic orbits of a circle diffeomorphism form
one pack.

If $\Nc$ is an interval, we can enlarge it and set $\tilde \Nc=3\Nc$. We can
also extend the map $f$ in the smooth way to a map $\tilde \Nc \to
\tilde \Nc$ so that no extra critical points are created. If $\Nc$ is a
circle, we set $\tilde \Nc=\Nc$.

Take $\Fc$ such that Propositions~\ref{co-excep}, \ref{co-poly-pd},
\ref{co-poly-cr} and Theorems \ref{thr:pb}, \ref{thr:cr} hold. Fix
some small $\kappa>0$.

If $\Fc$ is small enough, the number of critical points of maps in $\Fc$
is uniformly bounded. Hence, there can be only uniformly bounded
number of periodic orbits of $g\in \Fc$ which contain critical points in
their basins of attraction.

Fix a map $g : \tilde \Nc \to \tilde \Nc$ which is $C^\infty$ close to $f$
and let $E_j$, $j=1,\ldots,d_g$ be the critical intervals of $g$ given
by Proposition~\ref{co-excep}.  Let $\Or \subset \Nc$ be a periodic
orbit of $g$. Denote the orientation preserving period of $\Or$ by
$n$.

Let $p\in \Nc$ be a point of $\Or$ and define $T_p \subset \tilde \Nc$ be
a maximal interval containing $p$
such that each component of $T_p\setminus p$ contains at most one
point of $\Or$. Thus, the closure of $T_p$ contains five points of $\Or$
if $p$ is not one of the four points closest to the boundary of $\Nc$.

Now fix point $p \in \Or$ such that the corresponding interval $T_p$
has minimal length. Set $U_n=3T_p$.  Obviously, the interval $U_n$ is
a subset of $\tilde \Nc$ and the closure of $U_n$ can contain at most 13
points of the orbit $\Or$ while $U_n$ itself contains at most 11 points of
$\Or$ in its interior. Let $\{\hat U_k\}_{k=0}^n$ be a chain such that
$g^k(p) \in \hat U_k$ for all $k=0,\ldots, n$ and $\hat U_n=U_n$.

\begin{lemma}
  The intersection multiplicity of the chain  $\{\hat U_k\}_{k=0}^n$
  is bounded by 44.
\end{lemma}

This is almost the same as Lemma 10.3(i) in \cite{Melo1995}, p. 323,
where it is formulated for diffeomorphic pullbacks instead of the
chains. The proof is the same though.

\begin{proof}
  The interval $\hat U_n$ contains at most 11 points of the orbit
  $\Or$, hence $\hat U_k$ can contain at most 11 points of $\Or$ as
  well. Thus if an interval $\hat U_i$ contains a point $x$, there
  exist at most 10 points of $\Or$ between $g^i(p)$ and $x$.

  Suppose $x \in \hat U_{k_1} \cap \cdots \cap \hat U_{k_m}$ with
  $0\le k_1 <\cdots< k_m\le n$. Arguing as in the previous paragraph we
  see that $g^{k_i}(p)$ can be one of 22 points of $\Or$ around
  $x$. 
\end{proof}

By $U^l_n$ and $U^r_n$ we will denote the left
and right components of $U_n \setminus p$ and by $\{\hat U_k^l\}_{k=0}^n$ and
$\{\hat U_k^r\}_{k=0}^n$ the corresponding chains. Notice that the
point $g^k(p)$ is always a boundary point of the intervals $\hat
U_k^l$ and $\hat U_k^r$.

Let us inductively define intervals $U_k^r$, $k=0,\ldots,n$ by the
following rule. $U_k^r$ is the maximal interval containing $g^k(p)$ as
its boundary point and satisfying the following conditions:
\begin{itemize}
\item $g(U_k^r) \subset U_{k+1}^r$;
\item $g|_{U_k^r}$ is a diffeomorphism;
\item $|U_k^r|\le \kappa/2$;
\item if $g^k(p) \in 2E_j$ for some $j$, then $|U_k^r|\le  \kappa |E_j|/2$;
\item if $g^k(p) \not\in 2E_j$, then $U_k^r$ is disjoint from $E_j$.
\end{itemize}
Notice that $|U_k^r \cap E_j| \le \kappa |E_j|/2$ for all $k$ and $j$.

We will call $k$ a \emph{cutting time} if $g(U_k^r) \neq U_{k+1}^r$. The
cutting times can be one the following types:
\begin{itemize}
\item a \emph{critical cutting time} if $U_k^r$ contains a critical
  point of $g$ in its boundary;
\item an \emph{internal cutting time} if $|U_k^r|=\kappa/2$ or there
  exists a critical interval $E_j$ such that $g^k(p)\in 2E_j$ and
  $|U_k^r|=\kappa |E_j|/2$;
\item a \emph{boundary cutting time} if there exists a critical
  interval such that  $g^k(p)\not\in 2E_j$ and $U_k^r$ contains a
  boundary point of $E_j$ in its boundary.
\end{itemize}

Since the number of critical points and critical intervals of maps in
$\Fc$ is uniformly bounded and the intersection multiplicity of
$\{U_k^r\}_{k=1}^n$ is universally bounded, the number of critical and
boundary cutting times is uniformly bounded.

The intervals $U_k^l$ are defined in the same way. We also set
$U_k=U_k^l\cup U_k^r$.

By the definition of the intervals $U_k$ it follows that $g^n|_{U_0}$
is a diffeomorphism.  A simple argument shows that $U_0\subset T_p$,
see \cite{Melo1995}, Lemma 10.2, p. 322.

Now consider two cases. First, suppose that $g^n(U_0^r)$ is strictly
contained in $U_0^r$. In this case all periodic points in $U_0^r$
belong to the same pack and if a point in this interval is not
periodic, then it is in the attraction basin of one of the attracting
points of the pack.  Let $k_1$ be the minimal cutting time in
$\{U_k^r\}_{k=1}^n$. If $k_1$ is a critical time, then one of iterates of a
critical point is in $U_0^r$ and, therefore this critical point is in
the attraction basin of some periodic point in the pack. Since the
number of critical points of maps in $\Fc$ is uniformly bounded and
the same critical point cannot be in the attraction basins of two
different orbits, the number of such packs is uniformly bounded.
Similarly, if $k_1$ is a boundary cutting time, then a boundary point
of one of the critical intervals is in the attraction basin of a point
from the pack and the number of critical intervals is also uniformly
bounded. Now consider the case when $k_1$ is internal cutting time and
suppose that it corresponds to the critical interval $E_j$, i.e.
$g^{k_1}(p) \in 2E_j$ and $|U_{k_1}^r|=\kappa |E_j|/2$. Let $p'$ be
another periodic point and suppose that if we perform a similar
construction for $p'$, we get the first cutting time $k_1'$ internal
and $g^{n'}(U_0^{r'}) \subset U_0^{r'}$ where $U_{k}^{r'}$ are the
corresponding intervals.  Since every point in $U_k^{r}$ as well as in
$U_{k_1'}^{r'}$ is either periodic or its iterates are attracted to a
periodic orbit it follows that if the closures of these two intervals
have non empty intersection, then the points $p$ and $p'$ belong to
the same pack of periodic points. The interval $U_{k_1}^r$ has length
$\kappa |E_j|/2$, there are at most $2/\kappa + 1$ disjoint intervals
like this. If we take into account all critical intervals, then we see
that there can exist at most $(1+d_f)(1+2/\kappa)$ packs of periodic
orbits in this case.

Let us summarise. All periodic points $p$ such that $g^n(U^r_0)
\subset U^r_0$ belong to finite number of packs of periodic
orbits. The number of these packs is bounded by some constant which
depends on $\kappa$, $d_f$ and the number of critical points of maps
in $\Fc$ and does not depend on the choice of $g\in \Fc$.

From now on we will assume that $U_0^r \subset g^n(U_0^r)$.

Let $U_{-k}^r$ denote the diffeomorphic pullback of $U_0^r$ along the
orbit of $p$, $g^k(U_{-k}^r)=U_0^r$.

\begin{lemma}\label{lm:2pp}
  If the interval $U_{-n}^r$ contains another periodic point $p'$ with
  order preserving period $n' \le n$, then the periodic points $p$ and
  $p'$ belong to the same pack of periodic orbits.
\end{lemma}

\begin{proof}
  We know that the interval $U_0^r$ is subset of $T_p$, so $U_0^r$
  contains at most one point of $\Or$ in its interior. Let $q$ be this
  point if it exists, otherwise let $q=p$. If $q=p'$, we are done, so
  assume $q\neq p'$.

  Since $p'\in U_{-n}^r$ we get $g^{n}(p') \in U_0^r$ and, therefore,
  there are no periodic points from $\Or$ in the interval
  $(q,g^n(p'))$. Let $q'\in \Or$ be another periodic point from the
  orbit $\Or$ such that $p' \in (q,q')$ and the open interval $(q,q')$
  does not contain any points of $\Or$.

  If $g^{n'}(q)=q$ or $g^{n'}(q)=q'$, then the interval $[q,p']$ or
  $[p',q']$ is periodic and the points $q$ and $p'$ belong to the same
  pack of periodic orbits. Otherwise, the interval $(g^{n'}(q),
  g^{n'}(p'))=(g^{n'}(q), p')$ contains a point from the orbit $\Or$,
  and therefore, the interval $(q,g^n(p'))$ will contain a point from
  $\Or$ as $g^n:(q,p') \to (q,g^n(p'))$ is a diffeomorphism. This is a
  contradiction.
\end{proof}

\begin{proposition}
  There exist constants $\rho>0$ and $\kappa_0>0$ such that for any
  $\kappa \in (0,\kappa_0)$ there exists $M \in \N$ such that the following
  holds.

  For every $g\in \Fc$ there are at most $M$ exceptional packs of periodic orbits of
  $g$ such that if $\Or$ is not in one of the exceptional packs, then
  there is a point $\theta^r \in U_0^r$ such that 
  $$
  Dg^n(\theta^{r}) > 1+2\rho.
  $$
\end{proposition}

\begin{proof}
  We can assume that $Dg^n(x) <2$ for all $x \in U_0^r$, otherwise we
  have nothing to prove. Since $U_n=3T_p$ and $U_0 \subset T_p$ the
  closure of the interval $g^n(U_0^r)$ is contained in the interior of
  $U_n$. In particular, this implies that there exists at least one
  cutting time for $\{U_k^r\}_{k=1}^n$. Let $m$ be the minimal
  cutting time, i.e. there is no cutting time $m'$ with
  $m'<m$. Another property of the minimal cutting time is that
  $g^m(U_0^r) = U_m^r$.

  Consider several cases now. First, suppose that $m$ is critical or
  boundary cutting time. Let $M'= \frac 13 U_m^r$ and $M \subset
  U_0^r$ is a preimage of $M'$ under $g^m$. Due to the second part of
  Proposition~\ref{co-excep} we know that $B(g, U_m^r, M') > 1+\tau$
  where $\tau>0$ is a constant independent of choice of $g\in
  \Fc$. Moreover, $\tau$ does not change if we decrease $\kappa$. We
  know that $|g^k(U_0^r) \cap E_j| \le \kappa |E_j|/2$ for all
  $k=0,\ldots, n-1$ and $j$, so we can apply the first part of
  Proposition~\ref{co-excep} to maps $g^m: U_0^r \to U_m^r$ and
  $g^{n-m-1} : g(U_m^r) \to g^n(U_0^r)$. Decreasing $\kappa$ if
  necessary we can get $B(g^n, U_0^r, M) > 1+\tau/2$.

  Let $L$ and $R$ be connected components of $U_0^r\setminus M$ and
  let $L$ contain the point $p$ in its boundary. According to 
  Theorem~\ref{thr:cr} the interval $g^n(R)$ cannot be very small
  compared to the interval $g^n(U_0^r)$. Indeed, if $g^n(R)$ is small,
  then it has a huge space inside $U_n^r$, i.e. $Cg^n(R) \subset
  U_n^r$ for some large constant $C$. If we apply
  Theorem~\ref{thr:cr} to the map $g^{n-m} : g^m(R) \subset \hat U_m^r
  \to g^n(R) \subset U_n^r$, is  we can
  see that the interval $g^m(R)$ would have a big space in
  $U_m^r$. However, one of the components of $U_m^r \setminus g^m(R)$
  has length $2|g^m(R)|$, so the space is bounded. A similar argument
  holds for the interval $g^n(L)$, in this case we should consider
  $U_n^l \cup g^n(U_0^r)$ as a neighbourhood of $g^n(L)$.

  Thus, there exists a constant $\beta>0$ independent of the choice of
  $g\in \Fc$ such that
  $$
  |g^n(L)|>\beta |g^n(U_0^r)|
  \quad \mbox{and} \quad 
  |g^n(R)|>\beta |g^n(U_0^r)|.
  $$
  Since $U_0^r \subset g^n(U_0^r)$ and $Dg|_{U_0^r}<2$ it follows that
  $$
  |L|>\frac 12 \beta |U_0^r|
  \quad \mbox{and} \quad 
  |R|>\frac 12 \beta |U_0^r|.
  $$
  Now we can apply ``the First Expansion Principle'', see
  \cite{Melo1995}, Theorem 1.3, p. 280 to the map $g^n: U_0^r \to
  g^n(U_0^r)$ and get a point $\theta^r \in U_0^r$ with
  $Dg^n(\theta^r)>1+2\rho$ where $\rho$ does not depend on $g\in \Fc$.

  The remaining case we have to consider is when $m$ is the internal cutting
  time. By definition we know that in this case either
  $|U_m^r|=\kappa/2$ or there exists a
  critical interval $E_j$ such that $g^m(p) \in 2E_j$ and
  $|U_m^r|=\kappa |E_j|/2$. We will consider only the second case, the
  other one can be dealt with in the exactly same way.
  
  Consider an interval $U_{-6n}^r\subset U_0^r$. The derivative of
  $g^n$ on $U_0^r$ is bounded by 2, therefore $|U_{-6n}^r|>2^{-7}
  |g^n(U_0^r)|$ and the interval $U_0^r \setminus U_{-6n}^r$ has a definite space inside
  the interval $U_n^r$. Applying Theorem~\ref{thr:cr} to the map
  $g^{n-m}:\hat U_m^r \to U_n^r$ we get a constant $\gamma>0$ such that
  $|U_{-7n+m}^r| > \gamma |U_m^r|= \gamma \kappa |E_j|/2$. This constant
  is independent of $g\in \Fc$ and $\kappa$.

  Define an interval $W\subset U_m^r$ so that $g^m(p)$ is the boundary
  point of $W$ and $|W|=\frac 18 \gamma \kappa |E_j|$. Clearly, $g^{6n}|_W$ is
  a diffeomorphism.

  Let $\Or'$ be another periodic orbit of $g$ of period $n'$. Suppose that if for
  $\Or'$ we repeat the construction we did for $\Or$, then the
  corresponding first cutting time $m'$ is also of internal type with
  the same interval $E_j$. Let $W'$ be defined as $W$ but for the
  orbit $\Or'$.

  \underline{Claim.} If the intervals $W$ and $W'$ have a non empty
  intersection, then the orbits $\Or$ and $\Or'$ belong to the same pack.

  Without loss of generality we can assume that $n'\ge n$.  Also
  observe that the intervals $W$ and $W'$ have the same length.

  Let us consider several cases how the intervals $W$ and $W'$ can
  intersect. First, let us suppose that  $g^m(p) \in W'$. Then
  $g^{m+n'-m}(p) \in U_{-6n'}^{r'}$ and due to Lemma~\ref{lm:2pp} the
  points $p$ and $p'$ are in the same pack.

  Another case is $g^m(p) \not \in W'$ and $g^{m'}(p') \not \in W$. In
  this case it is easy to see that since $|U_{-7n'+m'}^{r'}|> 4 |W'|$ the
  interval $U_{-7n'+m'}^{r'}$ contains the point $g^m(p)$ and the same
  argument as above can be applied.

  The last case is $g^m(p) \not \in W'$ and $g^{m'}(p') \in W$. If
  $g^{m'+n}(p')=g^{m'}(p')$, then the interval $(g^m(p),g^{m'}(p'))$
  is periodic and the points $p$ and $p'$ are in the same pack. If 
  $g^{m'+n}(p')$ is in the interval $(g^m(p),g^{m'}(p'))$, then 
  $(g^m(p),g^{m'}(p'))$ is mapped into itself and iterates of the point
  $g^{m'}(p')$ are attracted to some periodic attractor which is
  impossible because $g^{m'}(p)$ is a periodic point.

  So, $g^{m'+n}(p')\not \in (g^m(p),g^{m'}(p'))$. This implies that
  $g^{m'+in}(p') \in U_m^r \setminus (g^m(p), g^{m'}(p'))$ for
  $i=1,\ldots,6$. The intervals $U_m^r$ and $U_{m'}^{r'}$ have the same
  length $\kappa |E_j|/2$, therefore, in this case 
  $$U_m^r \setminus (g^m(p), g^{m'}(p')) \subset U_{m'}^{r'}$$
  and the interval $g^{n'-m'}(U_{m'}^{r'})$ contains 6 points from the
  orbit $\Or'$ in its interior. This contradicts the fact that $U_n^r
  \supset g^{n'-m'}(U_{m'}^{r'})$ contains at most 5 points from $\Or'$.
  The claim is proved.

  Now we can finish the proof of the proposition. It follows from the
  claim that there are at most $8/(\gamma \kappa) + 1$ packs of
  periodic orbits such that a periodic point from such a pack can have
  the minimal cutting time of boundary type associated with the
  critical interval $E_j$. Since the number of the critical intervals
  is uniformly bounded the lemma follows.
\end{proof}

The theorem easily follows from this proposition.
Take $\kappa \in (0,\kappa_0)$ so small that
Proposition~\ref{co-excep} holds with
$$\epsilon=1 - \left(\frac{1+\rho}{1+2\rho}\right)^{\frac 13}$$
and $N=44$.

For this choice of $\kappa$ let $\theta^r\in U_0^r$ and $\theta^l\in U_0^l$ be
given by the proposition, so
$$
Dg^n(\theta^{r,l}) > 1+2\rho.
$$
Set $T=(\theta^l,\theta^r)$. Since $T\subset U_0$, $|g^k(U_0)|\le
\kappa$, and $|g^k(U_0) \cap E_i| \le \kappa |E_j|$ for all
$k=0,\ldots,n-1$ and $j$ and the map $g^n : T \to g^n(T)$ is a
diffeomorphism, Proposition~\ref{co-excep} can be applied to all
intervals $J^*\subset T^* \subset T$. We get
$$
B(g^n, T^*, J^*)^3 > \frac{1+\rho}{1+2\rho}.
$$
Now the ``Minimum principle'', \cite{Melo1995}, Theorem 1.1, p. 275,
can be applied and
$$
Dg^n(x)>1+\rho
$$
for all $x\in T$. In particular, $Dg^n(p)> 1+\rho$.

\section{Appendix}
\label{sec:appendix}

Following a referee suggestion we outline here proofs of
Lemma~\ref{lm:unifor} and uniform bounds on the Schwarzian derivative
used in the proof of Theorem~\ref{thr:sfn}.

\underline{Proof of Lemma~\ref{lm:unifor}} \newline
Suppose that the conclusion of the lemma is false. This means that
there exist
\begin{itemize}
\item a map $f\in C^1(\Nc)$,
\item a constant $\epsilon > 0$,
\item a sequence of maps $g_k \in C^1(\Nc)$, $k=1,2,\ldots$,
\item a sequence of intervals $I_k \subset \Nc$,
\item and a sequence of positive integers $n_k$
\end{itemize}
such that the following properties are satisfied:
\begin{enumerate}
\item $f$ does not have wandering intervals, \label{py:1}
\item $f$ does not have neutral periodic points, \label{py:2}
\item $|I_k|>\epsilon$ for all $k$,
\item $g_k \to f$ in $C^1$ norm as $k\to \infty$,
\item $n_k \to +\infty$,
\item $|g_k^{n_k}(I_k)| \to 0$, \label{py:to0}
\item $g_k^{n_k}(I_k)$ does not intersect an immediate attraction
  basis of a periodic attractor of $g_k$. \label{py:B}
\end{enumerate}

By considering a subsequence we can assume that the intervals $I_k$
converge to an interval $I_0$. This interval $I_0$ cannot be degenerate
as its length is bounded by $\epsilon$ from below. Notice that at this
point we cannot claim that $\liminf_{n \to +\infty} |f^n(I_0)| =0$.

\underline{Claim 1}
These are no periodic points of $f$ in $\interior(f^n(I_0))$,
$n=0,1,\ldots$, where $\interior$ denotes the interior of a set.

Indeed, if $a\in \interior(f^{n_0}(I_0))$ for some $n_0\ge 0$ is a
periodic point, then $a$ cannot be a neutral point of $f$ because of
Property~\ref{py:2}. Hence, under a small perturbation of $f$ the
point $a$ persists and there exists $k_0$ such that $g_k^{n_0}(I_k)$
contains a periodic point of $g_k$ for all $k\ge k_0$. If $a$ is an
attracting periodic point, then we get a contradiction with
Property~\ref{py:B}.  If $a$ is repelling, then there exists
$\epsilon_a>0$ such that $|f^n(I_0)|>\epsilon_a$ for all $n$. This
also holds for small perturbations of $f$, and it contradicts
Property~\ref{py:to0}.

Similarly one can proof

\underline{Claim 2}
Intervals $f^n(I_0)$ cannot have a non empty intersection with
immediate basins of attraction of attracting periodic points of $f$.

Let $W = \cup_{n=0}^\infty \interior(f^n(I_0))$. The set $W$ is not
necessarily forward invariant, but its closure is. Take a connected component $U$ of
$W$. If for some $m>0$ $f^m(U)\cap U \neq \emptyset$, then $f^m(U) \subset \bar U$,
where $\bar U$ is the closure of $U$.
Consider several cases:
\begin{enumerate}
\item If $U$ contains a periodic point of $f$, then one of $f^n(I_0)$
  contains a periodic point of $f$ in its interior. This contradicts
  Claim 1.
\item If $U$ is an interval and there are no periodic points of $f$ in
  $U$, then $f^m|_U$ is monotone and one of the boundary points $a$ of $U$
  is an attracting periodic  point of $f$. Moreover, the immediate
  basin of attraction of $a$ contains $U$, and therefore some
  $f^n(I_0)$ has a non empty intersection with it which is impossible
  according to Claim 2.
\item Let $U$ be a circle. In this case $\Nc=W=U$ and the map $f$ does
  not have periodic points. By compactness there are finitely many $0
  \le n_1< \ldots < n_r$ such that $\Nc=\cup_{i=1}^r
  \interior(f^{n_r}(I_0))$. It is easy to see that
  there exist $z \in \Nc$ and $l>0$ such that the points $z$ and
  $f^l(z)$ are in $\interior (f^{n_0}(I_0))$ for some $n_0$. Let $\epsilon_0 =
  \min_{x\in \Nc} |f^l(x)-x|$. Obviously $\epsilon_0>0$ as $f$ has no
  periodic points. Then $|g^l(x)-x| > \epsilon_0/2$ for all $x$ if $g$
  is sufficiently close to $f$. So, for large $k$ one has that
  $\{z,g_k^l(z)\} \subset g_k^{n_0}(I_k)$ and therefore
  $|g_k^n(I_k)|>\epsilon_0/2$ for all $n>n_0$. This contradicts
  Property~\ref{py:to0}.
\end{enumerate}

Finally, if the orbit of $U$ is disjoint, then either $U$ is a
wandering interval of $f$ or it is attracted to a periodic attractor. Both
cases are impossible because of Property~\ref{py:1} and Claim 2.
\hfill $\Box$

\vspace{2mm}

Let $c$ be a quadratic critical point of $f\in C^3(\Nc)$ and let
$B=Df^2(c)$. Fix a neighbourhood $\Fc \in C^3(\Nc)$ of $f$ and some
interval $T$ of $c$ so $f$ does not have other critical points in $T$.
We can assume that all maps in $\Fc$ have one quadratic critical point
in $T$.  If $\Fc$ and $T$ are small enough, we get $D^3g(x) Dg(x)-
\frac{3}{2}(D^2g(x))^2 < -B^2$ for all $g\in \Fc$ and $x \in T$.

Let $c_g\in T$ denote the critical point of $g\in \Fc$. Due to the
mean value theorem we get $Dg(x)=Dg(c_g)+D^2g(z)(x-c_g)$ for some $z\in
[c_g,x]$. Therefore, $|Dg(x)|<A|x-c_g|$ for some $A>0$ for all $g\in
\Fc$ and $x \in T$. Combining these inequalities we get
$$
Sg(x) =\frac{D^3g(x) Dg(x)-\frac 32(D^2g(x))^2}{(Dg(x))^2} < -\frac{B^2}{A^2|x-c_g|^2}.
$$
This is the required estimate. The rest of the proof of
Theorem~\ref{thr:sfn} literally follows the proof in
\cite{Kozlovski2000} or \cite{VanStrien2004}.

\bibliographystyle{alpha}
\bibliography{b3}

\end{document}